\documentclass[10pt,a4paper]{article}
\usepackage{amsmath,amsthm,amssymb,amscd}
\usepackage{graphicx,graphics}
\usepackage{verbatim}
\usepackage{mathrsfs}
\usepackage{enumerate}
\usepackage{color}
\usepackage[top=1in, bottom=1in, left=1in, right=1in]{geometry}

\newtheorem{theorem}{Theorem}[section]

\newtheorem{lemma}[theorem]{Lemma}

\newtheorem{corollary}[theorem]{Corollary}
\newtheorem{claim}[theorem]{Claim}
\newtheorem{conjecture}[theorem]{Conjecture}

\begin{document}

\title{Partially normal 5-edge-colorings of cubic graphs}

\vspace{3cm}
\author{Ligang Jin\footnotemark[1] ~and Yingli Kang\footnotemark[2]}
\footnotetext[1]{\small Department of Mathematics, Zhejiang Normal University. Yingbin Road 688, 321004 Jinhua, China. Email: \tt ligang.jin@zjnu.cn.}

\footnotetext[2]{\small Jinhua Polytechnic. Western Haitang Road 888, 321017 Jinhua, China. Email: \tt ylk8mandy@126.com.}

\maketitle

\begin{abstract}
{\small 
In a proper edge-coloring of a cubic graph, an edge $e$ is normal if the set of colors used by the edges adjacent to $e$ has cardinality 3 or 5.
The Petersen coloring conjecture asserts that every bridgeless cubic graph has a normal 5-edge-coloring, that is, a proper 5-edge-coloring such that all edges are normal.
In this paper, we prove a result related to the Petersen coloring conjecture.
The parameter  $\mu_3$ is a measurement for cubic graphs, introduced by Steffen in 2015. 
Our result shows that every bridgeless cubic graph $G$ has a proper 5-edge-coloring such that at least  $|E(G)|-\mu_3(G)$, which is no less than $\frac{4}{5}|E(G)|$, many edges are normal. 
This result improves on some earlier results of B\'{\i}lkov\'{a} and \v{S}\'{a}mal.}
\end{abstract}


\section{Instruction}
This paper focuses on Jaeger's Petersen coloring conjecture \cite{Jaeger_1988}, which states that every bridgeless cubic graph has a Petersen coloring.
This conjecture is stronger than Berge-Fulkerson conjecture, and also implies some other conjectures, such as 5-cycle double cover conjecture (shortly, 5CDCC).
There are several equivalent statements to the Petersen coloring conjecture, one of them is that every bridgeless cubic graph has a normal 5-edge-coloring.
However, only few results on this conjecture is known. Here, we follow \v{S}\'{a}mal's new approach \cite{Samal_2011} that might leads to a solution to this conjecture.
For a given bridgeless cubic graph, we look for a proper 5-edge-coloring yielding normal edges as much as possible.
In other words, we color the graph ``as normal as possible'' while the conjecture asserts that we can color the graph completely normal. 
The result of B\'{\i}lkov\'{a} \cite{Hana_MThesis} targets some classes of cubic graphs and shows that, we can color a generalized prism so that $\frac{2}{3}$ of the edges are normal, and we can color a cubic graph of large girth so that almost $\frac{1}{2}$ of the edges are normal.
In this paper, we prove that every bridgeless cubic graph $G$ has a proper 5-edge-coloring such that at least $|E(G)|-\mu_3(G)$ edges are normal. By a result of Kaiser, Kr\'{a}l and Norine in \cite{Kaiser2006}, it holds $\mu_3(G)\leq \frac{1}{5}|E(G)|$. Therefore, we can guarantee a proper 5-edge-coloring of $G$ containing at least $\frac{4}{5}|E(G)|$ normal edges, which improves these earlier results.

\subsection{Petersen coloring conjecture}
Given graphs $G$ and $H$, a mapping $\phi \colon\ E(G)\rightarrow E(H)$ is an $H$-coloring of $G$ if any three mutually adjacent edges of $G$ are mapped to three mutually adjacent edges of $H$.
The mapping $\phi$ is called a \textit{Petersen-coloring} if $H$ is the Petersen graph.

Jaeger \cite{Jaeger_1988} posed the following conjecture which would imply both Berge-Fulkerson Conjecture and 5-CDCC.

\begin{conjecture}[The Petersen coloring conjecture \cite{Jaeger_1988}]
	Every bridgeless cubic graph has a Petersen-coloring.
\end{conjecture}

This subsection devotes to some alternative formulations of the Petersen coloring conjecture.

Let $G$ be a graph. A set of edges $C$ is a \textit{binary cycle} if $C$ induces a subgraph of $G$ where every vertex has even degree.  
DeVos, Ne\v{s}et\v{r}il and Raspaud \cite{Devos_2006} defined that, given graphs $G$ and $H$, a mapping $\phi \colon\ E(G)\rightarrow E(H)$ is \textit{cycle-continuous} if the pre-image of each binary cycle of $H$ is a binary cycle of $G$. When $G$ and $H$ are cubic and additionally $H$ is cyclically 4-edge-connected,  
$G$ has a cycle-continuous mapping to $H$ if and only if $G$ has an $H$-coloring. This leads to the first alternate formulation of the Petersen coloring conjecture.
\begin{theorem}[e.g. \cite{Hana_MThesis}]
	A cubic graph has a Petersen-coloring if and only if it has a cycle-continuous mapping to the Petersen graph.
\end{theorem}
However, the study on cycle-continuous mapping makes no progress on solving the Petersen coloring conjecture so far.

Consider Cremona-Richmond configuration $G_{cr}$, which has 15 points and 15 lines, as drawn in Figure \ref{CR}. A \textit{CR-coloring} of a graph $G$ is a mapping from $E(G)$ to the points of $G_{cr}$ such that any three mutually adjacent edges of $G$ are mapped to three vertices of $G_{cr}$ that lie in a line.

\begin{figure}[ht]
	\centering
	\includegraphics[width=6cm]{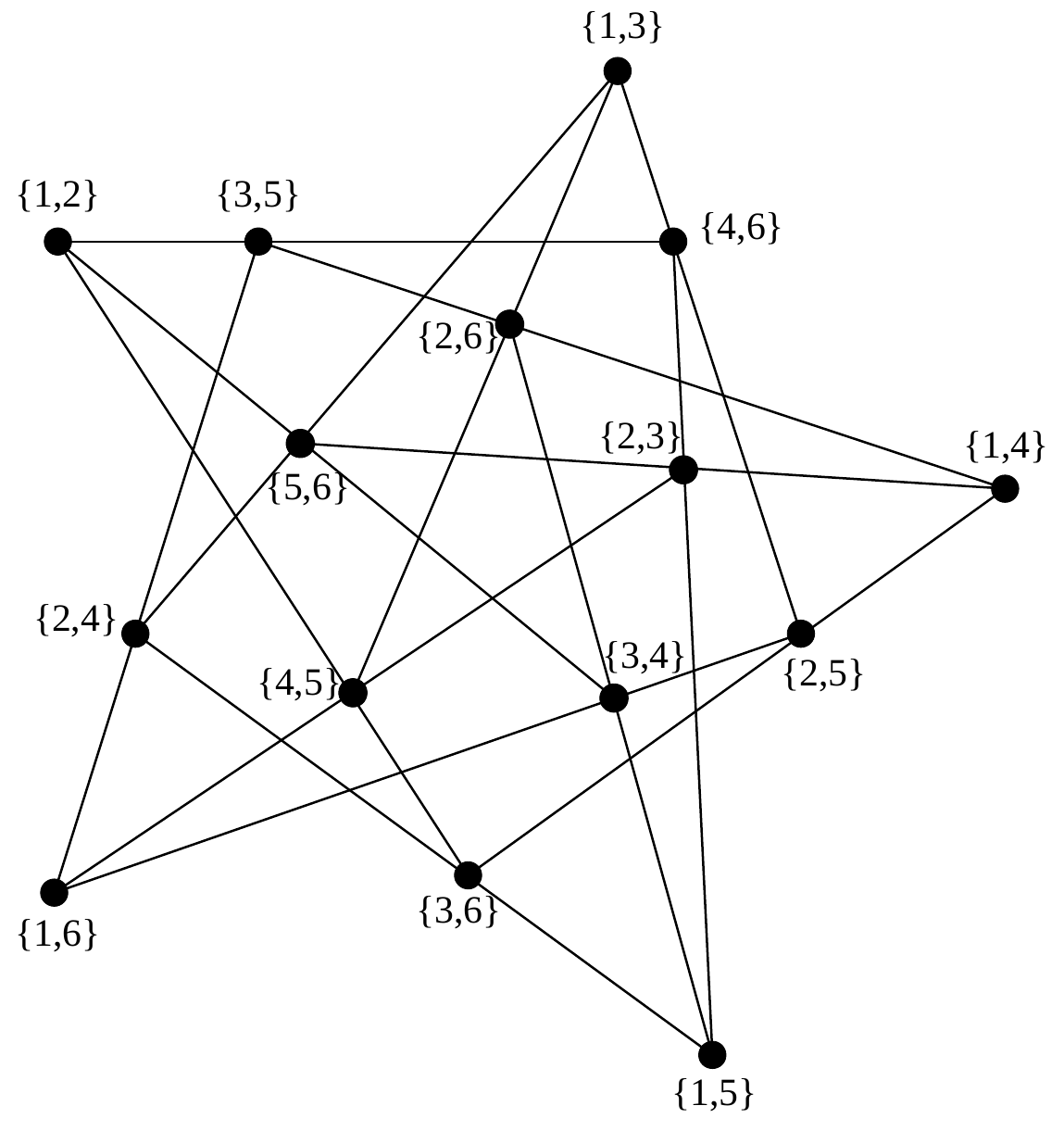}
	\caption{Cremona-Richmond configuration with $\{i,j\}$ labelling}\label{CR}
\end{figure}

\begin{theorem}[\cite{KMPRSS_2009}]\label{thm_BF_CR}
	A cubic graph has a Berge-Fulkerson cover if and only if it has a CR-coloring.
\end{theorem}

The truth of this theorem easily follows from a labelling of Cremona-Richmond configuration by $\{i,j\}$ with $1\leq i< j\leq 6$, as shown in Figure \ref{CR}. 
Here, we give another labelling of Cremona-Richmond configuration which yields that every CR-coloring of the graph $G$ is a nowhere-zero flow of $G$, that is, the flow values around a vertex sum up to zero. Such a labelling takes 15 non-zero elements of $\mathbb{Z}_2^4$, depicted in Figure \ref{CR2}.

\begin{figure}[ht]
	\centering
	\includegraphics[width=6cm]{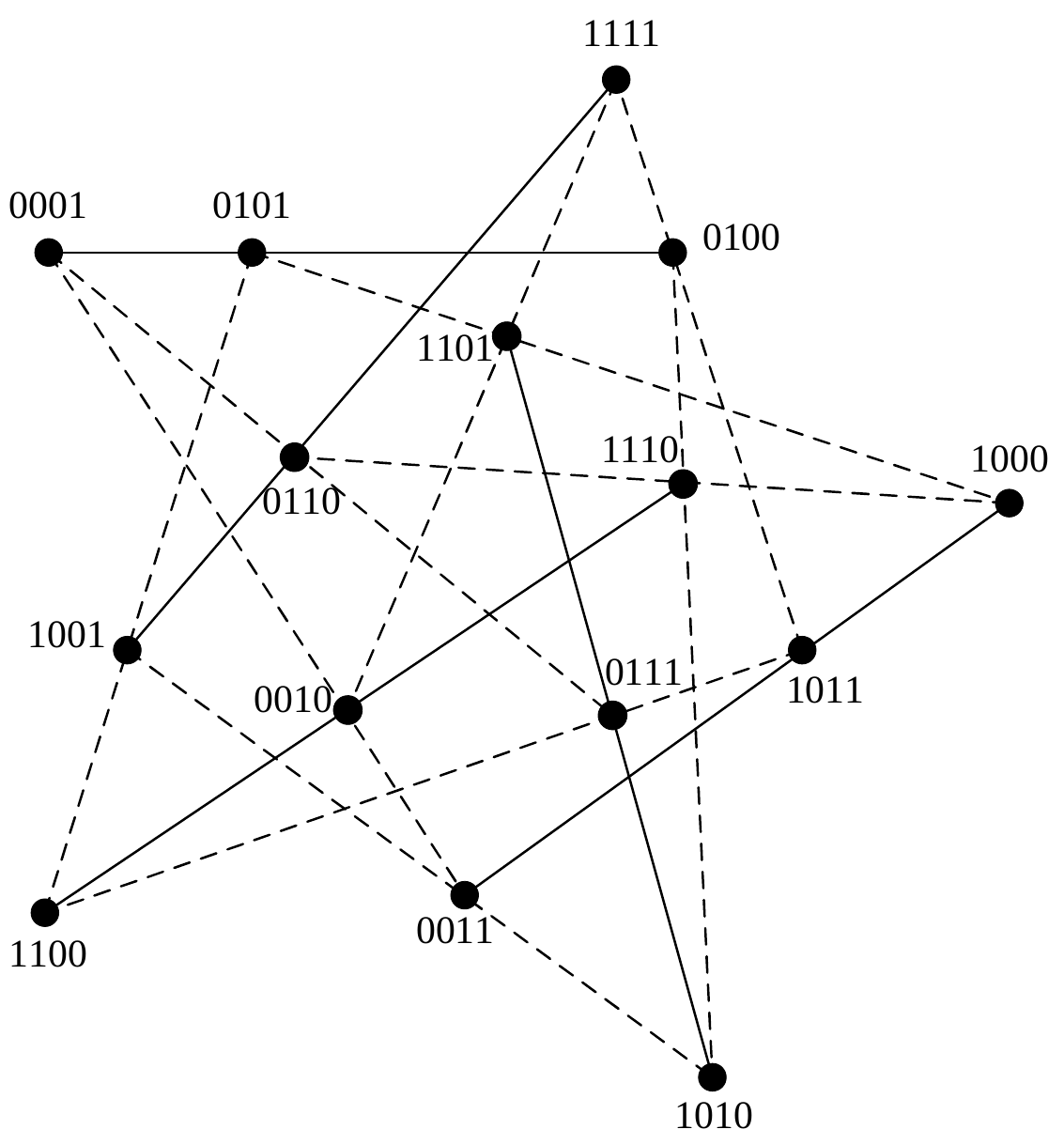}
	\caption{Cremona-Richmond configuration with $\mathbb{Z}_2^4$-labelling and with $L_{cr}$ in dotted line}\label{CR2}
\end{figure}

Let $L_{cr}$ be a set of 10 lines obtained from the lines of $G_{cr}$ by removing 5 pairwise disjoint lines. The dotted lines in Figure \ref{CR2} indicate an example of $L_{cr}$.
\begin{theorem}[\cite{KMPRSS_2009}]
	A cubic graph has a Petersen-coloring if and only if it has a CR-coloring using lines from $L_{cr}$.
\end{theorem}

From the previous two theorems, it is easy to see again that the Petersen coloring conjecture implies Berge-Fulkerson conjecture.

Unfortunately, the study on CR-colorings makes no progress on solving the Petersen coloring conjecture either.
Here, we focus on another alternative formulation of the Petersen coloring conjecture, in terms of normal 5-edge-colorings.

\subsection{Normal 5-edge-coloring}
Let $G$ be a cubic graph and $\phi\colon\ E(G)\rightarrow \{1,2,\ldots,5\}$ be a proper 5-edge-coloring.
An edge $e$ is \textit{poor} (or \textit{rich}) if $e$ together with its four adjacent edges uses precisely 3 (or 5) colors in total. An edge is \textit{normal} if it is either rich or poor, and it is \textit{abnormal} otherwise.
A \textit{normal 5-edge-coloring} is a proper 5-edge-coloring such that all the edges are normal.
Jaeger \cite{Jaeger_1985} showed the equivalence between Petersen colorings and normal 5-edge-colorings of a cubic graph.

\begin{theorem}[\cite{Jaeger_1985}]
	A cubic graph has a Petersen-coloring if and only if it has a normal 5-edge-coloring.
\end{theorem}

A possible minimal counterexample to the Petersen coloring conjecture was characterized in literatures.
Jaeger \cite{Jaeger_1988} proved that it must be a cyclically 4-edge-connected snark.
By the study on normal 5-edge-colorings of cubic graphs, H\"{a}gglund and Steffen \cite{Steffen_Hagglund_2014} showed that  
the minimal counterexample does not contain $K^*_{3,3}$ as a subgraph (see Figure \ref{K33} for $K^*_{3,3}$).

\begin{figure}[ht]
	\centering
	\includegraphics[width=4cm]{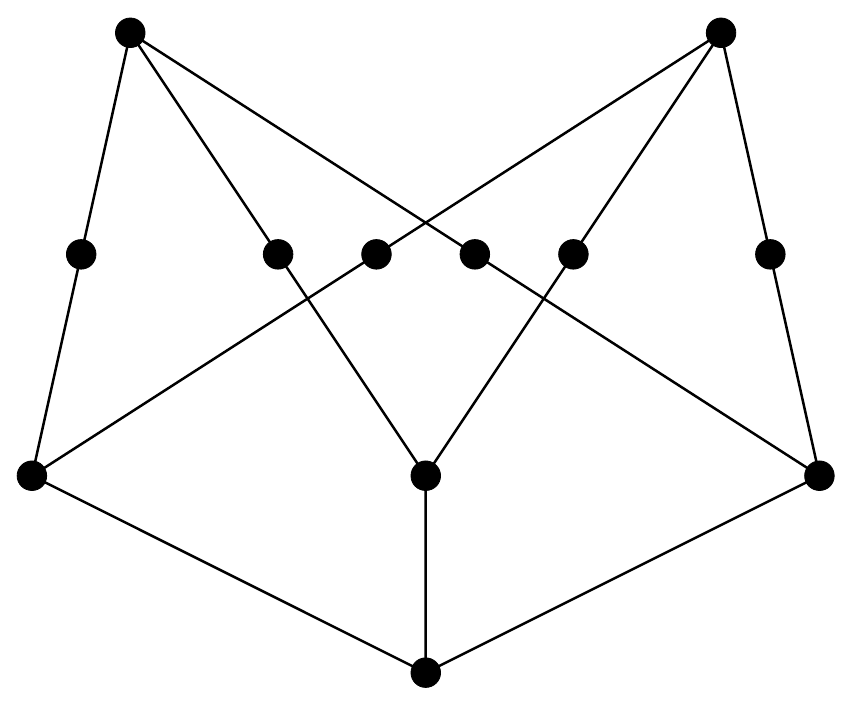}
	\caption{The graph $K^*_{3,3}$}\label{K33}
\end{figure}

A quite few classes of cubic graphs have been confirmed to have a normal 5-edge-coloring and thus a Petersen coloring as well. 
In \cite{Steffen_Hagglund_2014} it also showed that a cubic graph $G$ has a normal 5-edge-coloring if $G$ is a flower snark or a Goldberg snark or a generalized Blanu\v{s}a snark of type 1 or 2. With the aid of computer, Brinkmann et al. \cite{Brinkmann_Jan_2013} tested the Petersen coloring conjecture on cubic graphs of small order, and showed that every cubic graph of order no more than 36 has a normal 5-edge-coloring.
Recently, Ferrarini, Mazzuoccolo and Mkrtchyan \cite{Ferrarini_2019} confirm the existence of normal 5-edge-colorings for a family of Loupekhine snarks.

\subsection{Partially normal 5-edge-coloring}
Let $G$ be a cubic graph and $S_3$ be a list of three 1-factors $M_1, M_2, M_3$ of $G$. 
For $0\leq i\leq 3$, let $E_i$ be the set of edges that are contained in precisely $i$ elements of $S_3$.
Let $|E_0| = k$. 
The \textit{$k$-core} of $G$ with respect to $S_3$ (or to $M_1, M_2, M_3$)  is the subgraph $G_c$ of $G$
which is induced by $E_0 \cup E_2 \cup E_3$; that is, $G_c = G[E_0 \cup E_2 \cup E_3]$.
If the value of $k$ is irrelevant, then we
say that $G_c$ is a core of $G$.
Furthermore, $\mu_3(G) = \min \{ k \colon G \text{~has a $k$-core}\}.$
Clearly, every bridgeless cubic graph has a $\mu_3(G)$-core; and for any core, $E_0\cup E_2$ induces disjoint circuits.

Cores were introduced by Steffen \cite{Steffen_2014} recently and were used to prove partial results on some hard conjectures which are related to 1-factors of cubic graphs, such as Berge conjecture, Fan-Raspaud conjecture, and conjectures on cycle covers.
The parameter $\mu_3(G)$ can measure how far a cubic graph $G$ is from being 3-edge-colorable, and it was related to many other parameters, such as girth and oddness of $G$. We refer to \cite{Fiod_2018} for a survey on these kinds of measurements, and to \cite{Jin_Petersen-core,Jin_weak_core,Jin_fraction} for studies on cores and $\mu_3$.
In this paper, we will use them to prove a partial result on the Petersen coloring conjecture.

Considering that a normal 5-edge-coloring requires each edge to be normal, 
\v{S}\'{a}mal \cite{Samal_2011} presented a weaker problem approximate to the Petersen coloring conjecture, that is, to search for a proper 5-edge-coloring such that the normal edges are as many as possible. Here, such a coloring is called a partially normal 5-edge-coloring. Later on, B\'{\i}lkov\'{a} proved that a generalized prism has a proper 5-edge-coloring with two third of the edges normal (\cite{Hana_MThesis}, Theorem 2.3) and a cubic graph of large girth has a proper 5-edge-coloring with approximately half of the edges normal (\cite{Hana_MThesis}, Theorem 3.6).
In this paper, we show that for every bridgeless cubic graph, there exists a proper 5-edge-coloring such that almost all the edges are normal.
More precisely, we prove the following theorem.

\begin{theorem}\label{thm_main}
	Every bridgeless cubic graph $G$ has a proper 5-edge-coloring such that at least  $|E(G)|-\mu_3(G)$ many edges are normal.
\end{theorem}

Since $\mu_3(G)\leq \frac{1}{5}|E(G)|$ by a result of Kaiser, Kr\'{a}l and Norine in \cite{Kaiser2006}, a direct consequence of this theorem is as follows.
\begin{corollary}
	Every bridgeless cubic graph $G$ has a proper 5-edge-coloring such that at least  $\frac{4}{5}|E(G)|$ edges are normal.
\end{corollary}
 
The proof of this theorem will be done by constructing such a proper 5-edge-coloring with the help of the structural properties on cores. First of all, we need some definitions and lemmas.

\section{Definitions and lemmas}
Let $G$ be a cubic graph. If $C$ is a circuit of $G$, then $\langle C\rangle$ denotes the set of edges not on $C$ but having at least one end on $C$. Analogously, if $P$ is a path of $G$ with ends $x$ and $y$, then $\langle P\rangle$ denotes the set of edges not on $P$ but having at least one end on $P-x-y$. If $H$ is a set of vertex-disjoint circuits or paths of $G$, then define that $\langle H\rangle = \bigcup_{h\in H} \langle h\rangle.$

Let $G$ be a cubic graph and $X\subseteq E(G)$.
Let $\psi\colon\ X\rightarrow \{1,\ldots,5\}$ be a proper edge-coloring of $G[X]$. 
Let $H$ be a set of vertex-disjoint circuits or paths of $G$.
The subgraph $H$ is \textit{$\psi$-extendable} if the following three items hold: (i) $E(H)\cap X=\emptyset$; (ii) $\psi(e)\in \{1,2,3\}$ for $e\in \langle H \rangle \cap X$; (iii) we can assign $E(H)\cup \langle H\rangle \setminus X$ with colors from $\{1,2,3\}$ so that the resulting coloring remains proper.
Applying the third item is called \textit{$\psi$-extending} $H$.
Let $v$ be a vertex of $G$, $E(v)$ be the edges incident with $v$ and $C(v)$ be the colors on $E(v)$, that is, $C(v)=\{\psi(e)\colon\ e\in X\cap E(v)\}$.
An edge $h$ of $E(v)$ is \textit{$\psi$-good} on $v$ if either $C(v)=\{1,2,3\}$ or $e\notin X$ and $C(v)=\{4,5\}$.
Let $H$ be a subgraph of $G$ of minimum degree 2. Define $\mathcal{E}_\psi(H)$ to be edges of $G-E(H)$ that has an end on $H$ and is not $\psi$-good on this end. If $\psi$ is clear from the context, we write $\mathcal{E}(H)$ for short.

Let $G_c$ be a core of a cubic graph $G$ with respect to three 1-factors $M_1,M_2,M_3$.
The \emph{major-coloring} of $G$ with respect to $M_1,M_2,M_3$ (or to $G_c$) is a mapping $\phi\colon\ E(G)\setminus E(G_c)\rightarrow \{1,2,3\}$ defined as $\phi(e)=i$ for each $e\in (E(G)\setminus E(G_c))\cap M_i$.
Let $k$ be an integer and $k\geq 2$.
A \emph{string} $P$ of $G_c$ is a subgraph of $G_c$ consisting of distinct odd circuits $C_0,C_1,\ldots,C_k$ of $G[E_0\cup E_2]$ and edges $e_1,\ldots,e_k$ of $E_3$ such that each $e_i$ connects a vertex $u_i$ of $C_{i-1}$ to a vertex $v_i$ of $C_i$. 
Such a string is denoted by $C_0e_1C_1\ldots e_kC_k$ or $C_0(u_1v_1)C_1\ldots (u_kv_k)C_k$.
The two circuits $C_0$ and $C_k$ are called \emph{end-circuits} of $P$, and the remaining circuits are called \emph{middle-circuits} of $P$.

Let $G_c$ be a core of a cubic graph $G$ and $\phi_m$ be the major-coloring of $G$ with respect to $G_c$.
Let $P^1,\ldots,P^s$ be pairwise disjoint strings of $G_c$ by notation $P^i=C_0^i(u_1^iv_1^i)C_1^i\ldots (u_{t_i}^iv_{t_i}^i)C_{t_i}^i$ for $i\in \{1,\ldots,s\}.$
Denote by $Q$ the union of all the odd circuits of $G[E_0\cup E_2]$ not contained in any of these strings.
The union of $P^1,\ldots,P^s$ is a \emph{wave} if for $i\in \{1,\ldots,s\}$ there exist a path $p^i_j$ of $C^i_j$ between $v^i_j$ and $u^i_{j+1}$ for $j\in \{1,\ldots,t_i-1\}$, a path $p^i_0$ of $C^i_0$ with $u^i_1$ as an end and a path $p^i_{t_i}$ of $C^i_{t_i}$ with $v^i_{t_i}$ as an end, satisfying the following three items: 
\begin{enumerate}[(1)]
	\item $P=\{p^i_j\colon 1\leq i\leq s \text{ and } 0\leq j\leq t_i\}$ is $\phi_m$-extendable.
	\item For any $i\in \{1,\ldots,s\}$ and $j\in\{0,t_i\}$, if $|E(p^i_j)|\leq |E(C^i_j)|-2$, let $\overline{p}^i_j$ consist of $p^i_j$ and the end-edge of $C^i_j-E(p^i_j)$ that is not incident with $u^i_1$ or $v^i_{t_i}$, and let $P'$ be obtained from $P$ by constituting $\overline{p}^i_j$ for $p^i_j$, then $P'$ is not $\phi_m$-extendable.
	\item For any component $q_1$ of $Q$ and another component $q_2$ of $P\cup Q$, we have $\langle q_1\rangle\cap \langle q_2\rangle \cap E_3=\emptyset.$
\end{enumerate}
Such a wave is denoted by $P^1+\cdots+P^s$.

\begin{lemma}\label{lem_wave}
	Let $G_c$ be a core of a bridgeless cubic graph $G$. If $G_c$ has a string, then it has a wave.	
\end{lemma}

\begin{proof}
	We construct such a wave $W$ by an algorithm. 
	
	Let $\phi_m$ be the major-coloring of $G$ with respect to $G_c$, and let $H$ be the union of odd circuits of $G[E_0\cup E_2]$.
	Since $G_c$ has a string, say $s$, we can take two circuits $C_u$ and $C_v$ and an edge $e$ of $s$ such that
	$e$ connects a vertex $u$ of $C_u$ to a vertex $v$ of $C_v$. 
	Initialize $W$ to be a graph consisting of $C_u,C_v$ and $e$. Initialize $P$ and $F$ to be empty sets, which will collect paths.
	
	(*) If there exists no $\phi_m$-extendable path $p$ on $C_u$ whose one end is $u$ and the other end (say $w$) is connected to a vertex (say $x$) of some circuit (say $C_x$) of $H-W$ such that $wx\in E_3$, then add into $P$ and $F$ the longest $\phi_m$-extendable path on $C_u$ which takes $u$ as an end; otherwise (i.e., if such $p$ exists), we do the following: take such $p$ of minimum length, let $W$ include $wx$ and $C_x$, add $p$ into $P$ and $F$, and then repeat this argument with $x$ and $C_x$ instead of $u$ and $C_u$ respectively until no such $p$ exists anymore. 
	
	We can see from (*) that for any two components $p_1$ and $p_2$ of $F$ (w.l.o.g., assume that $p_1$ was put into $F$ earlier than $p_2$), we have $\langle p_1\rangle\cap \langle p_2\rangle \cap E_3=\emptyset$ by the length minimality of $p_1$. Since no two edges of $E_3$ are adjacent, we can further deduce that $\langle p_1\rangle\cap E_3$ and $\langle p_2\rangle\cap E_3$ are disjoint. Therefore, we can $\phi_m$-extend $F$ and we do it. Denote by $\phi'_m$ the resulting coloring.
 	Repeat the argument (*) with $v, C_v, \phi'_m$ instead of $u, C_u, \phi_m$, respectively.
 	We can see that the resulting $F$ is still $\phi_m$-extendable.
	Now the first string of $W$ is completed.
	
	If $G_c$ has a string disjoint with $W$, then reset $F$ to be an empty set and apply the same argument on this string as on $s$, which yields the second string of $W$.
	Repeat this until $G_c$ has no strings disjoint with the resulting $W$.
	
	Now the construction of $W$ is completed. We shall prove that $W$ is a wave. Following the notation in the wave definition, let $Q=H-W$. Firstly, 
	since $G_c$ has no strings disjoint with $W$ right now,
	for any two distinct components $q_1$ and $q_2$ of $Q$, we have $\langle q_1\rangle\cap \langle q_2\rangle \cap E_3=\emptyset.$
	For any component $p$ of $P$, by the length maximality of $p$ when it belongs to an end-circuit of $W$ and by the length minimality of $p$ when it doesn't, we can deduce that $\langle p\rangle\cap \langle q\rangle \cap E_3=\emptyset$ for any component $q$ of $Q$.
	Therefore, the item (3) holds for $W$.
	Secondly, for any two components of $P$ locating in different strings, take $p$ as the one put into $P$ earlier than the other (say $q$) and again, by the length maximality or minimality of $p$ we can deduce that $\langle p\rangle\cap \langle q\rangle \cap E_3=\emptyset$.
	Moreover, the algorithm shows that each component of $P$ is $\phi_m$-extendable.
	Therefore, $P$ is $\phi_m$-extendable as well, i.e., the item (1) holds for $W$.
	It remains to show that the item (2) holds for $W$. If not, then following the notation of item (2),
	$W$ has an end-circuit $C^i_j$ as described in item (2) but $P'$ is $\phi_m$-extendable, contradicting the length maximality of $p^i_j$.
\end{proof}

Let $G_c$ be a core of a cubic graph $G$. Let $D$ be a circuit of $G[E_0\cup E_2]$. Define $\sigma(D)$ to be the number of vertices of $D$ incident with an edge from $E_3$. Note that $\sigma(D)\geq |\langle D \rangle \cap E_3|$. Define $\Omega(G_c)=\{C\colon\ C$ is a circuit of $G[E_0\cup E_2]$ such that $\sigma(C)=1$ and $|E(C)|\leq 5\}$. Let $C_1$ and $C_2$ be two distinct circuits of $\Omega(G_c)$.
$C_1$ and $C_2$ are \textit{$G_c$-connected} if there is an edge from $E_1$ connecting a vertex of $C_1$ to a vertex of $C_2$.
Let $e_i$ be the unique edge from $\langle C_i \rangle \cap E_3$ for $i\in \{1,2\}$.
Let $X\subseteq E(G-E(C_1)-E(C_2))$ and $\psi\colon\ X\rightarrow \{1,\ldots,5\}$ be a proper edge-coloring of $G[X]$. $C_1$ and $C_2$ are \textit{$\psi$-connected} if $e_1$ and $e_2$ are adjacent to a common edge $e\in X$ such that $\psi(e)\in \{4,5\}$.

Let $X$ be a set of edges of a cubic graph $G$ and let $\psi\colon\ X\rightarrow \{1,\ldots,5\}$ be a proper edge-coloring of $G[X]$.
An edge $e$ of $G$ is \textit{$\psi$-inner} if $e$ together with its adjacent edges belong to $X$; otherwise, $e$ is \textit{$\psi$-outer}.
Let $G_c$ be a core of $G$. Define $\theta_{G_c,\psi}$ as a function on $E(G)$ given by
\begin{equation*}
	\text{for each $\psi$-inner edge $e$, $\theta_{G_c,\psi}(e)=$}
	\begin{cases}
		1 & \text{if $e\in E_0$ and $e$ is normal,} \\
		-1 & \text{if $e\notin E_0$ and $e$ is abnormal,} \\
		0 & \text{otherwise;}	
	\end{cases}
\end{equation*}
\begin{equation*}
	\text{and for each $\psi$-outer edge $e$, $\theta_{G_c,\psi}(e)=$}
	\begin{cases}
		0 & \text{if $e\in E_0$}, \\
		-1 & \text{if $e\notin E_0$}.
	\end{cases}
\end{equation*}
If $G_c$ and $\psi$ are clear from the context, we write $\theta$ for short.
Moreover, for $X\subseteq E(G)$, define $\theta(X)=\sum_{x\in X}\theta(x)$.
We write $\theta(H)$ short for $\theta(E(H))$ for a subgraph $H$ of $G$.

A direct consequence of the definition of $\theta$ is the following lemma.
\begin{lemma}\label{lem_mu3_abnormal}
	Let $G_c$ is a $k$-core of a cubic graph $G$.
	If $\psi\colon\ E(G)\rightarrow \{1,\ldots,5\}$ is a proper edge-coloring of $G$,
	then $G$ has $k-\theta(G)$ abnormal edges.
\end{lemma}

Now we are ready to prove the main theorem of this paper.

\section{Proof of Theorem \ref{thm_main}}
Trivially, the theorem holds true for 3-edge-colorable cubic graphs.
We may assume that $G$ is not 3-edge-colorable.
Let $G_c$ be a $\mu_3(G)$-core of $G$ and let $\phi_m$ be the major-coloring of $G$ with respect to $G_c$.
Let $H=G[E_0\cup E_2]$ and denote by $H_1$ the graph consisting of all the even circuits of $H$. 
If $G_c$ has a string, then it has a wave $W$ by Lemma \ref{lem_wave}, and denote by $H_2$ the graph consisting of all the odd circuits of $H$ that are contained in $W$; otherwise, to be convenient, we say that $W$ and $H_2$ are empty graphs. 
Let $H_3=H-H_1-H_2$. 
We will extend $\phi_m$ to a proper 5-edge-coloring $\phi_m'$ of $G$ by coloring $H_1,H_2,H_3$ in order and simultaneous the edges of $E_3$.
By Lemma \ref{lem_mu3_abnormal}, to show that the final coloring $\phi_m'$ yields at most $\mu_3(G)$ edges abnormal, it suffices to prove $\theta_{G_c,\phi_m'}(G)\geq 0$.
In what follows, since $G_c$ is fixed and we always consider the current coloring extended from $\phi_m$,  
we write $\theta$ and $\mathcal{E}$ briefly.   
Let $\mathcal{K}$ be a set initialized to be empty.
We will use $\mathcal{K}$ to collect subgraphs of $G$ which receive colors during the extension of $\phi_m$.

For each circuit $C$ of $H_1$, assign $E(C)$ with colors 4 and 5 alternately along $C$.
For each $e\in E(C)$, by the definition of the function $\theta$, if $e\in E_0$ then $\theta(e)\geq 0$. If $e\in E_2$, then $e$ is adjacent to two edges of the same color from $\{1,2,3\}$, so $e$ is poor yielding $\theta(e)=0$. Therefore, $\theta(C)\geq 0=|\mathcal{E}(C)|.$ Add $C$ into the set $\mathcal{K}$.

To describe the structure of the wave $W$, we use same notations as in the definition of a wave. 
By Property $(1)$ in the wave definition, we can $\phi_m$-extend $P$ and we do it.
Notice that the remaining part of $W$ are disjoint paths. Color them with 4 and 5 alternately along each path.
Add $W$ into $\mathcal{K}$.

\begin{claim}
	For each string $P^i$ of $W$, we have $\theta(P^i) \geq 2=|\mathcal{E}(P^i)|$.
\end{claim}

\begin{proof}
	Let $t_i=d$.
	We will show that if $d\geq 2$ then $\theta(u^i_jv^i_j)+\theta(C^i_j)\geq 1$ for each $j\in \{1,\ldots,d-1\}$.
	Firstly, for each $e\in E(C^i_j)\setminus E(p^i_j)$, we have $e\in E_0\cup E_2$. Again, by the definition of the function $\theta$, if $e\in E_0$ then $\theta(e)\geq 0$, and if $e\in E_2$ then $e$ is poor yielding $\theta(e)=0$. Therefore, $\theta(C^i_j)-\theta(p^i_j)\geq 0$. 
	Moreover, since each edge of $p^i_j$ is either rich or poor, $\theta(p^i_j)=|E(p^i_j)\cap E_0|$.
	Since the value $\theta$ of an edge is at least -1, the conclusion $\theta(u^i_jv^i_j)+\theta(C^i_j)\geq 1$ holds true, provided that $|E(p^i_j)\cap E_0|\geq 2$.
	Hence, we may next assume that $|E(p^i_j)\cap E_0|\leq 1$.
	It follows that $p^i_j$ is just an edge from $E_0$. So we could choose a $\phi_m$-extension of $p^i_j$ so that $u^i_jv^i_j$ is poor. The conclusion holds as well.	

	We next show that $\theta(C^i_0)\geq 1$, while the equality holds only if $C^i_0$ is a triangle.
	Denote by $x$ an end-vertex of $p^i_0$ rather than $u^i_0$, and by $a$, $b$ and $c$ the edges incident with $x$ such that $a\in E(p^i_0)$ and $b\in \langle C^i_0 \rangle$.
	Since $p^i_0$ is a path, $|E(C^i_0)|-|E(p^i_0)|\geq 1$.
	We distinguish two cases.
	
	Case 1: assume that $|E(C^i_0)|-|E(p^i_0)|>1$.	
	Firstly, for each $e\in E(C)\setminus E(p^i_0+c)$, again by the definition of the function $\theta$, if $e\in E_0$ then $\theta(e)\geq 0$, and if $e\in E_2$ then $e$ is poor yielding $\theta(e)=0$. 
	Hence, $\theta(C^i_0-E(p^i_0)-c)\geq 0.$
	Secondly, by Properties (1) and (2) of the wave definiton, $P$ is $\phi_m$-extendable but $P+\{c\}$ is not. Hence, we can deduce that all the colors $1,2,3$ appear on the adjacent edges of $c$, yielding that $c$ is rich and $c\in E_0$. Thus, $\theta(c)=1.$ 
	Finally, since each edge of $p^i_0$ is either rich or poor except the edge $a$,
	we have $\theta(p^i_0)=|E(p^i_0)\cap E_0|-1$. Therefore, we can conclude that 	$\theta(C^i_0)\geq |E(p^i_0)\cap E_0|$.
	By again the length maximality of $p^i_0$, we can deduce that $\langle p^i_0+ c \rangle$ uses at least two kinds of colors.
	It follows that $|E(p^i_0)\cap E_0|\geq 2$ and so, $\theta(C^i_0)\geq 2$.
	
	Case 2: assume that $|E(C^i_0)|-|E(p^i_0)|=1$.
	Now $c$ and $p^i_0$ together form the circuit $C^i_0$.
	Notice that both $a$ and $c$ might be neither rich nor poor.
	We have $\theta(C^i_0)=\theta(c)+\theta(p^i_0)\geq |E(C^i_0)\cap E_0|-2$. Hence, the conclusion holds, provided that $|E(C^i_0)\cap E_0|\geq 4$. We may next assume that $|E(C^i_0)\cap E_0|\leq 3$. It follows that $C^i_0$ is of length either 5 or 3. If $C^i_0$ is of length 5, then $|E(C^i_0)\cap E_0|= 3$. W.l.o.g, see Figure \ref{C5} for the coloring of $C^i_0\cup \langle C^i_0\rangle$, which yields $\theta(C^i_0)=2,$ we are done.
	\begin{figure}[hh]
		\centering
		\includegraphics[width=10cm]{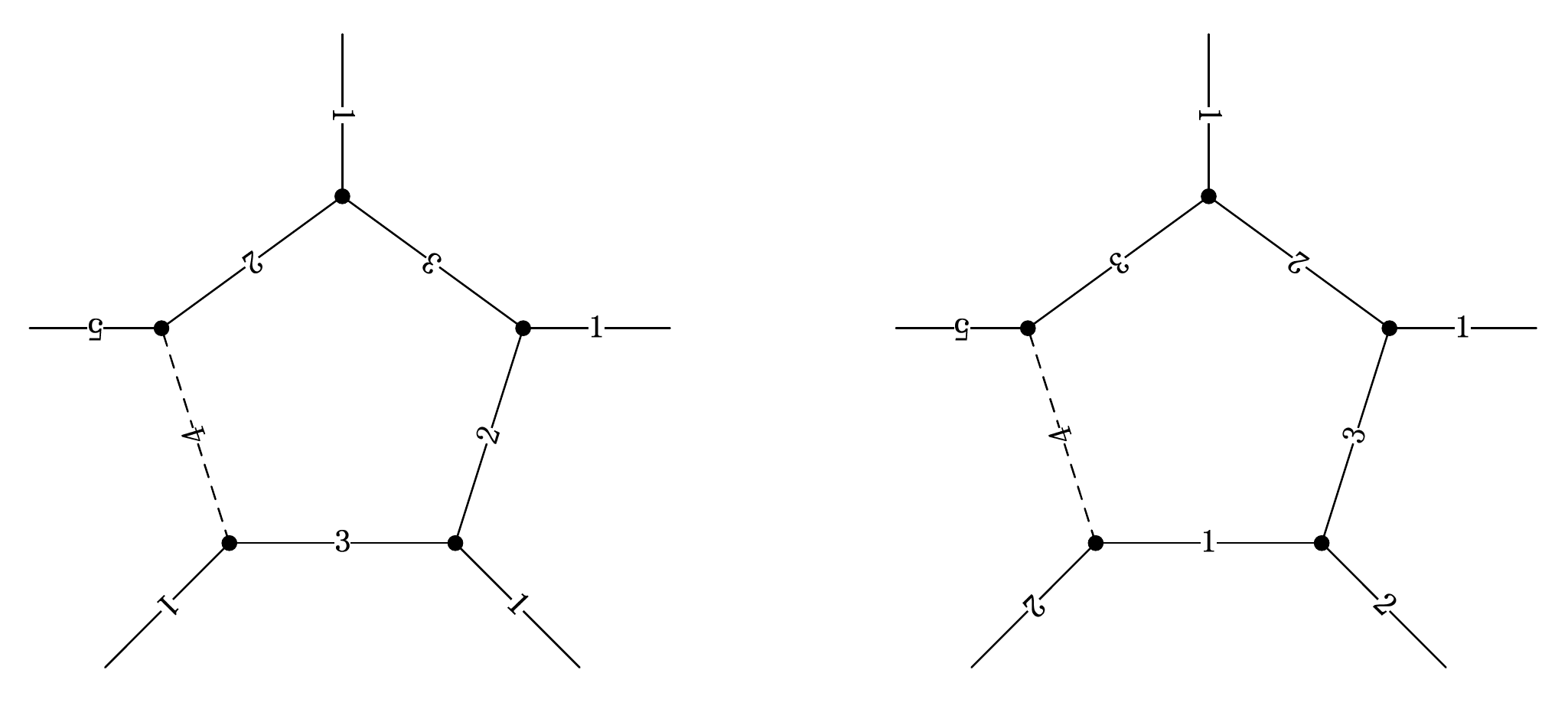}\\
		\caption{A coloring of $\langle C^i_0 \rangle \cup E(C^i_0)$ in two cases. Case 1 (left): $\langle p^i_0 \rangle$ uses one color; case 2 (right): $\langle p^i_0 \rangle$ uses at least two colors.}
		\label{C5}
	\end{figure}
	We may next assume that $C^i_0$ is of length 3, i.e., it is a triangle. 
	It follows that $|E(C^i_0)\cap E_0|\in \{2,3\}$.
	If $|E(C^i_0)\cap E_0|=3$, then $\theta(C^i_0)\geq 1,$ we are done. 
	If $|E(C^i_0)\cap E_0|=2$, without loss of generality, see Figure \ref{C3} for the coloring of $C^i_0\cup \langle C^i_0\rangle$, which yields $\theta(C^i_0)=1,$ we are done as well.
	\begin{figure}[hh]
		\centering
		\includegraphics[width=4cm]{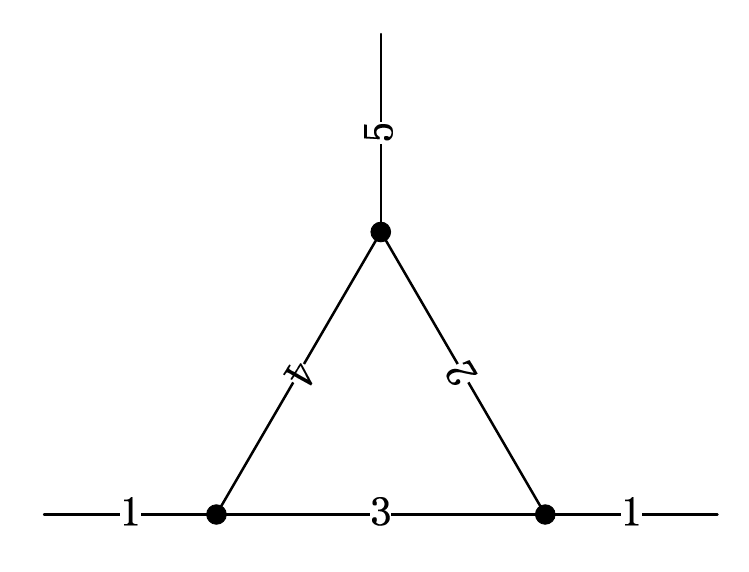}\\
		\caption{A coloring of $C^i_0$.}
		\label{C3}
	\end{figure}
	
	Similarly, we can prove that $\theta(C^i_d)\geq 1$, while the equality holds only if $C^i_d$ is a triangle.
	
	Now we are ready to calculate $\theta(P^i)$, given by 
	\begin{equation}
	\theta(P^i)=
	\theta(C^i_0)+\sum^{d-1}_{j=1}(\theta(u^i_jv^i_j)+\theta(C^i_j))+\theta(u^i_dv^i_d)+\theta(C^i_d)\geq 1+(d-1)-1+1=d\geq 1, \label{eq_string}
	\end{equation}  
	while the equality $\theta(P^i)=1$ holds only if $d=1$ and both $C^i_0$ and $C^i_d$ are triangles.
	Hence, to prove $\theta(P^i)\geq 2$, it suffices to consider the equality case.
	In this case, we can take a $\phi_m$-extension of $P$ so that the edge $u_dv_d$ is poor.
	Following the equation (\ref{eq_string}), we have $\theta(P^i)\geq 2$.
		
	Moreover, let $b'$ be an edge of $\langle C^i_d \rangle$ denoted in a similar way as $b$. From the coloring of $P^i$, it is easy to see that $\mathcal{E}(P^i)=\{b,b'\}.$ So, $|\mathcal{E}(P^i)|=2$. 
\end{proof}

Let $H_3''=H_3 \cap \Omega(G_c)$ and  $H_3'=H_3-H_3''$.
We will color $H_3'$ and $H_3''$ in order.  

For each circuit $C$ of $H_3'$, we add $C$ into $\mathcal{K}$, and we will color $E(C)$ so that $\theta(C)\geq |\mathcal{E}(C)|.$
If $C$ is $\phi_m$-extendable, then $\phi_m$-extend $C$ and consequently, $\theta(C)=|E(C)\cap E_0| > 0= |\mathcal{E}(C)|$.	   
Let us next assume that $C$ is not $\phi_m$-extendable.
Take the longest $\phi_m$-extendable path $q$ on $C$ such that $E_3\cap \langle q\rangle\neq \emptyset$.
Denote by $e_1$ and $e_2$ the two end-edges of $q$ and by $e_i'$ the edge of $E(C)\setminus E(q)$ that is adjacent to $e_i$ for $i\in \{1,2\}$.
Since $C$ is not $\phi_m$-extendable, $|E(C)\setminus E(q)|\geq 1.$ We distinguish three cases.

Case 1: assume that $|E(C)\setminus E(q)|>1$. We $\phi_m$-extend $q$ and assign $E(C)\setminus E(q)$ with colors 4 and 5 alternately. By the choice of $q$, all of the colors $1,2,3$ appear on the adjacent edges of $e_1'$, yielding that $e_1'$ is rich and belongs to $E_0$. Thus $\theta(e_1')=1$. Similarly, we can deduce that $\theta(e_2')=1$. Moreover, since $E_3\cap \langle q\rangle\neq \emptyset$, it follows that $|E(q)\cap E_0|\geq 2$.
Hence, $\theta(C)\geq |E(q)\cap E_0|+\theta(e_1')+\theta(e_2')+\theta(e_1)+\theta(e_2)\geq 2=|\mathcal{E}(C)|.$

Case 2: assume that $|E(C)\setminus E(q)|=1$ and $|E(C)\cap E_0|\geq 5$.
In this case, $e_1'$ and $e_2'$ are the same edge.
We $\phi_m$-extend $q$ and assign $e_1'$ with the color 4. So, $\theta(C)\geq |E(C)\cap E_0|-3\geq 2=|\mathcal{E}(C)|.$

Case 3: assume that $|E(C)\setminus E(q)|=1$ and $|E(C)\cap E_0|\leq 4$.
Clearly, both $|E(C)|$ and $\sigma(C)$ are odd. Since $\sigma(C)\leq |E(C)\cap E_0|$ always holds true,
we have $\sigma(C)\in \{1,3\}$.
If $\sigma(C)=3$, then by the equality $|E(C)|+\sigma(C)=2|E(C)\cap E_0|$, we have $|E(C)|\in \{3,5\}$. Moreover, the property (3) of the wave definition implies that all edges of $\langle C \rangle \cap E_3$ are uncolored. Hence, we can deduce that $C$ is $\phi_m$-extendable, a contradiction.
We may next assume that $\sigma(C)=1$.
By the same equality as above, $|E(C)|\leq 7$.
Moreover, since $C\notin \Omega(G_c)$, we have $|E(C)|\geq 7.$
Hence, $|E(C)|=7$.
Denote by $f$ the unique edge of $E_3\cap \langle C\rangle$.
We proceed in two subcases according to the colors $\langle C \rangle$ receives.

Subcase 3.1: assume that $\langle C \rangle$ uses at most two kinds of colors from $\{1,2,3\}$, say the colors 1 and 2. Assign $f$ with the color 3 and its two adjacent edges on $C$ with the colors 4 and 5, respectively. The remaining edges of $C$ can be properly assigned with colors from $\{1,2,3\}$. One can directly calculate from the coloring that $\theta(C)\geq 2=|\mathcal{E}(C)|.$

Subcase 3.2: assume that $\langle C \rangle$ uses all the colors $1,2,3$. Without loss of generality, see the left of Figure \ref{C7} for the coloring of $\langle C \rangle$. We extend the coloring to $E(C)$ and $f$ as depicted in the right of Figure \ref{C7}. By a direct calculation, $\theta(C)\geq 2=|\mathcal{E}(C)|.$

\begin{figure}[ht]
	\centering
	\includegraphics[width=9cm]{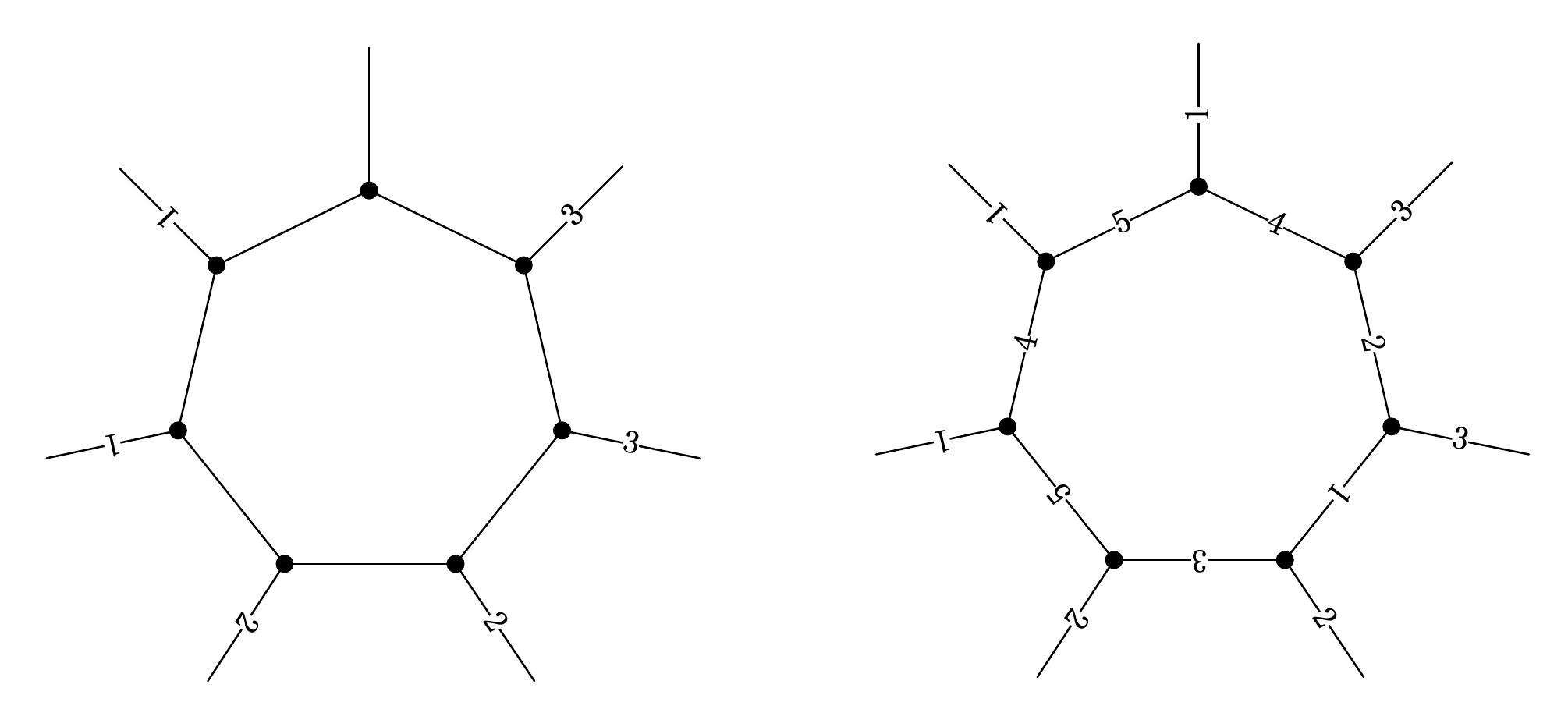}
	\caption{A coloring of the circuit $C$ for subcase 3.2} \label{C7}
\end{figure}

To complete the coloring of $H$, it remains to color the edges of $H_3''$.
Let $\phi_1$ be the current coloring extended from $\phi_m$.
We will color first all pairs of uncolored $\phi_1$-connected circuits and then all pairs of uncolored $G_c$-connected circuits and finally the remaining uncolored circuits of $H_3''$.

Let $C'$ and $C''$ be a pair of uncolored $\phi_1$-connected circuits of $H_3''$. 
Say that $C'=[u'_1\ldots u'_{k'}]$ and $C''=[u''_1\ldots u''_{k''}]$ with $u_1'x', u_1''x''\in E_3$.
Clearly, $k',k''\in\{3,5\}.$
Let $y'$ and $y''$ be the third neighbors of $u_2'$ and $u_2''$, respectively. By Property $(3)$ of the wave definition, $u_1'x'$ and $u_1''x''$ are uncolored, and $x'x''$ has color either 4 or 5.
Assign $u_1'x'$ and $u_1''x''$ with the color of $x'x''$.
Take $\alpha\in \{1,2,3\}\setminus \{\phi_1(u_2'y'),\phi_1(u_2''y'')\}$, and with the color $\alpha$ we assign $u_1'u_2'$ and $u_1''u_2''$ and reassign $x'x''$. Let $\phi_2$ be the resulting coloring.
$\phi_2$-extend the longest $\phi_2$-extendable paths of the form $u_1'u_2'\ldots$ or $u_1''u_2''\ldots$, and properly color the remaining uncolored edges on $C'\cup C''$ with 4 or 5.

If $x'x''\in E(W)$, then let $C_x$ be the component of $W$ containing $x'x''$; otherwise, $x'x''$ is contained in a circuit of $H_1\cup H_3'$, and let $C_x$ be this circuit.
Let $\mathcal{C}$ be the graph consisting of $C_x$, circuits $C'$ and $C''$, and edges $u_1'x'$ and $u_1''x''$.
We substitute $C_x$ for $\mathcal{C}$ in $\mathcal{K}$ and will show that $\theta(\mathcal{C})\geq |\mathcal{E}(\mathcal{C})|$.

We first prove that $\theta(C'),\theta(C'')\geq 1$. 
W.l.o.g., let $\alpha=3$. Recall that $k'\in \{3,5\}$.
If $k'=3$, then from the coloring extension of $C'$ as shown in Figure \ref{C3r}, a direct calculation gives $\theta(C')\geq 1$.
If $k'=5$, then $\langle C'\rangle$ uses either one or two kinds of colors from $\{1,2,3\}$. W.l.o.g., see Figure \ref{C5r} for the coloring extension in three cases. For each case, we can calculate that $\theta(C')\geq 1$ as well.
Similarly, we can prove $\theta(C'')\geq 1$.
\begin{figure}[ht]
	\centering
	\includegraphics[width=7cm]{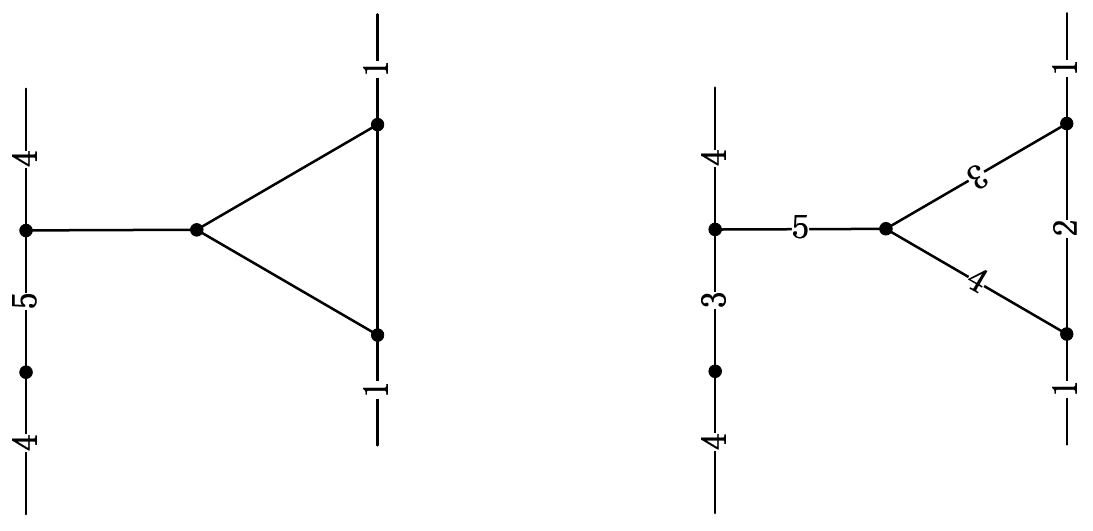}
	\caption{A coloring extension of $C'$ when it has length 3} \label{C3r}
\end{figure}

\begin{figure}[ht]
	\centering
	\includegraphics[width=8cm]{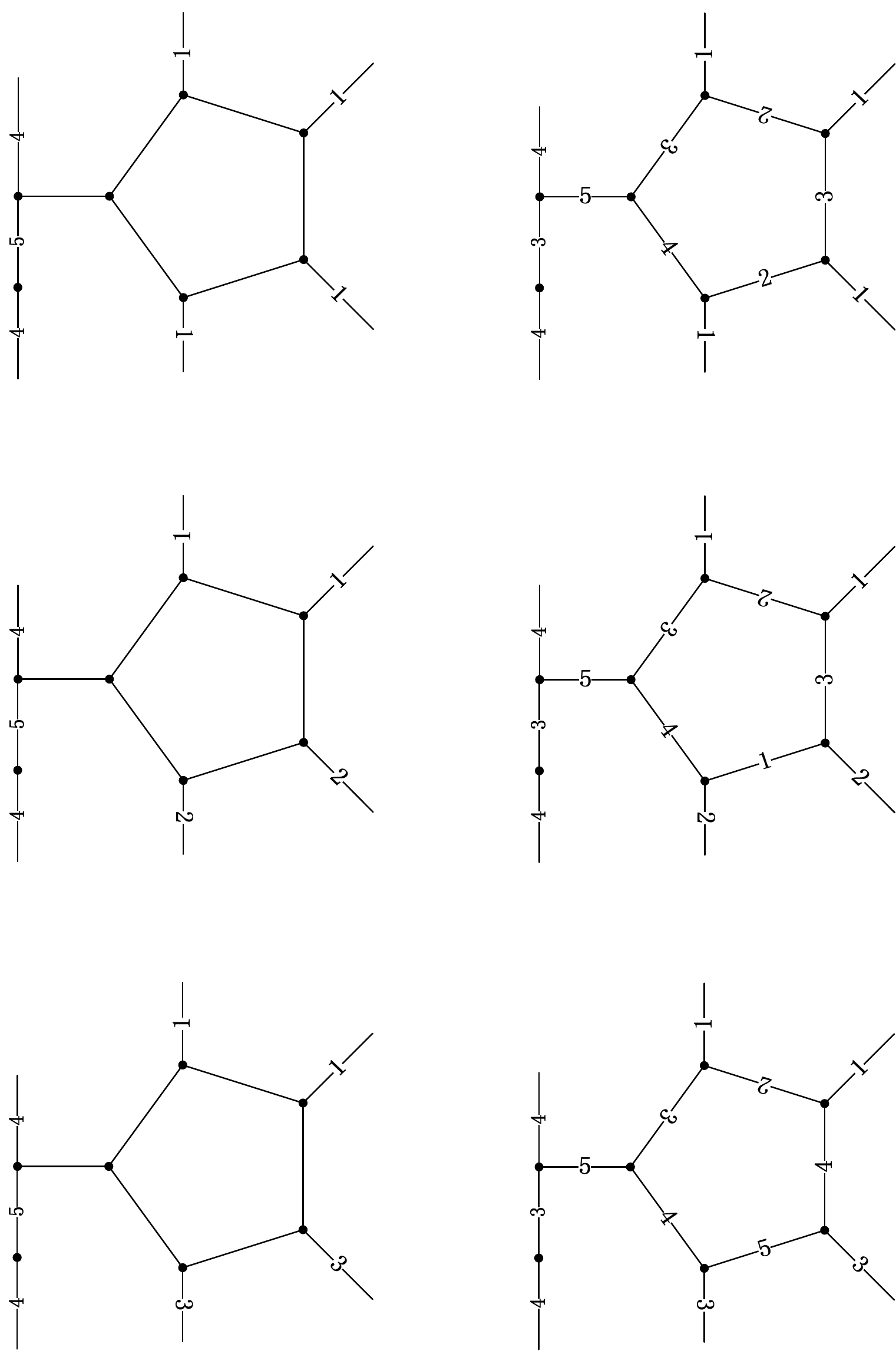}
	\caption{A coloring extension of $C'$ when it has length 5} \label{C5r}
\end{figure}

Denote by $a$ and $b$ the values of $\theta(C_x)$ and $|\mathcal{E}(C_x)|$ before $C'\cup C''$ receives colors, respectively.
We already have the conclusion that $a\geq b.$
Note that $x'u_1'$ and $x''u_1''$ are uncolored edges before $C'\cup C''$ receives colors. 
So by the definition of the function $\theta$,
the coloring of $C'\cup C''$ does not decrease the value $\theta$ of $x'x''$ and of its two adjacent edges locating on $C_x$, and does make $x'u_1'$ and $x''u_1''$ poor.
It follows that $\theta(C_x)\geq a$ and $\theta(x'u_1')=\theta(x''u_1'')=0$.
Thus, $\theta(\mathcal{C})= \theta(C_x)+\theta(C')+\theta(C'')+\theta(x'u_1')+\theta(x''u_1'')\geq a+2.$
Moreover, $\mathcal{E}(\mathcal{C})$ contains two more edges than $\mathcal{E}(C_x)$ (might be the same edge).
Hence, $|\mathcal{E}(\mathcal{C})|\in\{b+1,b+2\}$.
Now we can see that $\theta(\mathcal{C})\geq |\mathcal{E}(\mathcal{C})|$.

Let $C'$ and $C''$ be a pair of uncolored $G_c$-circuits of $H_3''$.
Choose an edge $e\in \langle C'\rangle \cap \langle C'' \rangle \cap E_1$.
Let $\mathcal{C}$ be the graph consisting of $C',C''$ and $e$.
Add $\mathcal{C}$ into $\mathcal{K}$.
If both $C'$ and $C''$ are triangles, then remove the color of $e$ and denote by $\phi_3$ the resulting coloring, and then we can $\phi_3$-extend $\mathcal{C}$. In this case, $\theta(\mathcal{C})=4>\mathcal{E}(\mathcal{C})=0,$ we are done.
So, we may next assume that $C'$ is of length 5. Reassign the color of $e$ with 4 and still denote by $\phi_3$ the resulting coloring.
If $C'$ is $\phi_3$-extendable, then we $\phi_3$-extend it;
otherwise, we can $\phi_3$-extend $C'-e'$, where $e'$ is an edge of $C'$ adjacent to $e$. 
Do the same to the cycle $C''$. Finally, assign $e'$ and $e''$ with the color 5 if they exist.
By the resulting coloring, a direct calculation gives $\theta(\mathcal{C})\geq |\mathcal{E}(\mathcal{C})|$.

We can see that $\theta(k)\geq |\mathcal{E}(k)|$ holds for each component $k$ of $\mathcal{K}$.
By taking the sum over $k$, we have 
\begin{equation}\label{eq_theta_k_final}
\theta(\mathcal{K})=\sum_{k\in \mathcal{K}}\theta(k)\geq \sum_{k\in \mathcal{K}}|\mathcal{E}(k)|\geq |\mathcal{E}(\mathcal{K})|.
\end{equation}

Let $T$ be all the remaining uncolored circuits of $H_3''$. 
To complete the coloring $\phi_m'$ of the whole graph $G$, we will first color all the uncolored edges in $E_3\setminus \langle T\rangle$, and then color $T$ and $E_3\cap \langle T\rangle$. 

For each uncolored edge $e$ of $E_3\setminus \langle T\rangle$, the four edges adjacent to $e$ are already colored. We can properly assign $e$ with a color from $\{1,2,\ldots,5\}$. Denote by $\phi_4$ the resulting coloring.

Let $\overline{T'}=G-T\cup \langle T\rangle$.
We will show that $\theta(\overline{T'})\geq 0$.
Let $e$ be an edge of $\overline{T'}-\mathcal{K}\cup \mathcal{E}(\mathcal{K})$
and $u$ be an end of $e$.
If $u$ locates on $T\cup \mathcal{K}$, since $e\notin \langle T\rangle\cup \mathcal{E}(\mathcal{K})$, $u$ locates on $\mathcal{K}$ and $e$ is $\phi_4$-good on $u$;
otherwise, since $T\cup \mathcal{K}$ has the same vertex set as $G_c$, the three edges around $u$ receive colors by $\phi_m$ and stand with them during all the previous coloring extensions, yielding that $e$ is $\phi_4$-good on $u$ as well.
By the choice of $e$ and $u$, we can conclude that 
\begin{equation}\label{eq_theta_T}
\theta(\overline{T'}-\mathcal{K}\cup \mathcal{E}(\mathcal{K}))\geq 0.
\end{equation} 
Moreover, notice that $T\cup \langle T\rangle$ and $\mathcal{K}\cup \mathcal{E}(\mathcal{K})$ may have common edges, which apparently belong to $\langle T\rangle\cap \mathcal{E}(\mathcal{K})$.
Therefore,
$\theta(\overline{T'})=\theta(\overline{T'}-\mathcal{K}\cup \mathcal{E}(\mathcal{K}))+\theta(\mathcal{K}) +\theta(\mathcal{E}(\mathcal{K})-\langle T\rangle\cap \mathcal{E}(\mathcal{K}))\geq 0+|\mathcal{E}(\mathcal{K})|-|\mathcal{E}(\mathcal{K})-\langle T\rangle\cap \mathcal{E}(\mathcal{K})|\geq 0,$
where the first inequality follows by equations (\ref{eq_theta_k_final}) and (\ref{eq_theta_T}) and the fact that the value $\theta$ of an edge is at least $-1$.
 
It remains to color $T$ and $E_3\cap \langle T\rangle$.
For each circuit $C$ of $T$, we will color $C$ so that  $\theta_{\phi_5}(\overline{T'}\cup C \cup \langle C \rangle)\geq 0$ for the resulting coloring $\phi_5$.
Say that $C$ is of length $k$ and of vertices $u_1,\ldots,u_k$ in cyclic order.
For $1\leq i\leq k$, denote by $v_i$ the neighbor of $u_i$ not on $C$.
Recall that $C\in \Omega(G_c)$. So, $k\in \{3,5\}$ and $\langle C\rangle \cap E_3$ contains precisely one edge, say $e=u_1v_1$.
Let $e_1$ and $e_2$ be other two edges around $v_1$.
Property $(3)$ of the wave definition implies that $e$ is still uncolored and that $e_1$ and $e_2$ are of colors 4 and 5.
If $e_1$ is adjacent to an uncolored edge $e'$ rather than $e$, then $e'\in E_3\cup \langle C' \rangle$ for some $C'\in T$, yielding that $C$ and $C'$ are $\phi_1$-connected circuits of $T$, a contradiction.  
Hence, $e$ is the only uncolored edge adjacent to $e_1$ or $e_2$.
For $i\in\{1,2\}$ let $\gamma_i$ be the color making $e_i$ normal if $e$ receives it.
Such $\gamma_i$ always exists and $\gamma_i\in \{1,2,3\}$.
We distinguish two cases.

Case 1: assume that $\langle C\rangle$ uses one same color, say the color 1.

Subcase 1.1: assume that at least one of $\gamma_1$ and $\gamma_2$ is not color 1, say $\gamma_1=2$. 
If $k=3$, then assign the edges $e,u_1u_2,u_2u_3,u_3u_1$ with colors $2,4,3,5$ respectively; otherwise, assign $e,u_1u_2,u_2u_3,u_3u_4,u_4u_5,u_5u_1$ with $2,4,3,2,3,5$, respectively. 
Since the coloring of $e$ makes $e_1$ from a $\phi_4$-outer edge to a normal edge, it increases the value $\theta(e_1)$ by 1. Moreover, a direct calculation gives $\theta(C\cup \langle C \rangle)\geq -1$. Therefore,  $\theta_{\phi_5}(\overline{T'}\cup C \cup \langle C \rangle)\geq 0$ holds.

Subcase 1.2: assume that $\gamma_1=\gamma_2=1$ and that not both $v_2$ and $v_k$ are incident with edges of color 2 and of color 3. W.l.o.g., let $v_2$ be incident with no edges of color 2. Reassign $u_2v_2$ with color 2 and consequently, we can $\phi_4$-extend $C$.
Since $T$ contains no $G_c$-connected circuits, the color reassigning of $u_2v_2$ makes no changes to the coloring of $\langle C \rangle$ for any other choice of $C$.
On one hand, a direct calculation gives $\theta(C\cup \langle C \rangle)=2$ if $k=3$, and $\theta(C\cup \langle C \rangle)=3$ otherwise.
On the other hand, reassigning $u_2v_2$ might decreases $\theta$ on other two edges around $v_2$ by at most 1 for each. 
Therefore, $\theta_{\phi_5}(\overline{T'}\cup C \cup \langle C \rangle)\geq 0$ holds.

Subcase 1.3: assume that $\gamma_1=\gamma_2=1$ and that both $v_2$ and $v_3$ are incident with edges of color 2 and of color 3. If $k=3$, then reassign $v_2u_2$ and $v_3u_3$ with color 4 and color 5, respectively, and assign $e,u_1u_2,u_2u_3,u_1u_3$ with colors $1,5,1,4,$ respectively; 
otherwise, reassign $u_2v_2$ with color $4$ and assign $e,u_1u_2,u_2u_3,u_3u_4,u_4u_5,u_5u_1$ with colors $1,5,3,2,3,2$, respectively.
For the former case, we increase the value $\theta$ of $e_1$ and $e_2$ by 1 for each and decrease $\theta$ of the other four edges around $v_2$ or $v_3$ by at most 1 for each. Moreover, $\theta(C\cup \langle C \rangle)=2$ by a direct calculation. Hence,  $\theta_{\phi_5}(\overline{T'}\cup C\cup \langle C \rangle)\geq 0$. For the latter case, we increase the value $\theta$ of $e_1$ and $e_2$ by 1 for each and decrease $\theta$ of the other two edges around $v_2$ by at most 1 for each. Moreover, $\theta(C\cup \langle C \rangle)=0$. Hence,  $\theta_{\phi_5}(\overline{T'}\cup C\cup \langle C \rangle)\geq 0$ holds as well.

Case 2: assume that $\langle C\rangle$ uses two kinds of colors, say the colors 1 and 2.
It follows that $k=5$.
Assign $e$ the same as $u_2v_2$ and assign 
$u_1u_2,u_2u_3,$
$u_3u_4,u_4u_5,u_5u_1$ with colors $4,5,4,3,5$, respectively.
A direct calculation gives $\theta(C\cup \langle C \rangle)\geq 0$. 
Therefore, $\theta_{\phi_5}(\overline{T'}\cup C\cup \langle C \rangle)\geq 0$.

Now we complete the coloring $\phi_m'$ of $G$ such that $\theta_{\phi_m'}(G)\geq 0$.
We are done with the proof of the theorem.

\section{Acknowledgment}
The authors are grateful to professor Eckhard Steffen for his helpful discussion on the first draft of the paper, and also to an annonymous referee for pointing out a gap of the proof of a lemma in an earlier version of the paper.
The first author is supported by NSFC (Grant number: 11801522).
The second author is supported by NSFC (Grant number: 11901258).

\end{document}